\documentclass[10]{amsart}
\usepackage{latexsym,amssymb,amsmath,amsthm,amsopn,graphics,xy,epsfig,picture,epic}
\usepackage{color}
\def\cc{{\mathcal C}}

\def\ff{{\mathcal F}}

\def\mm{{\mathcal M}}

\def\pp{{\mathcal P}}

\def\tt{{\mathcal T}}

\def\ffi{\varphi}
\def\eps{\varepsilon}
\def\dst{\displaystyle}

\DeclareMathOperator{\dist}{dist}
\DeclareMathOperator{\vect}{span}

\def\N{{\mathbb{N}}}

\def\Q{{\mathbb{Q}}}
\def\R{{\mathbb{R}}}

\def\Z{{\mathbb{Z}}}

\newcommand{\norm}[1]{{\left\|{#1}\right\|}}
\newcommand{\ent}[1]{{\left[{#1}\right]}}
\newcommand{\abs}[1]{{\left|{#1}\right|}}
\newcommand{\scal}[1]{{\left\langle{#1}\right\rangle}}

\newtheorem{lemma}{Lemma}[section]
\newtheorem{proposition}[lemma]{Proposition}
\newtheorem{theorem}[lemma]{Theorem}

\theoremstyle{definition}
\newtheorem{conjecture}{Conjecture}

\newtheorem{definition}{Definition}
\newtheorem{remark}{Remark}
\newtheorem{example}{Example}
\date{\today}

\newcounter{rep}
\newcommand{\rep}[1]{
		}

\newcounter{rea}
\newcommand{\rea}[1]{
		}

\newcounter{res}
\begin{document}

\title{Density of the span of powers of a function \`a la M\"untz-Sz\'asz}

\author{Philippe Jaming \& Ilona Simon}

\address{Philippe Jaming
\noindent Address: Univ. Bordeaux, IMB, UMR 5251, F-33400 Talence, France.
CNRS, IMB, UMR 5251, F-33400 Talence, France.}
\email{Philippe.Jaming@gmail.com}

\address{Ilona Simon
\noindent Address: Institute of Mathematics and Informatics, University of P\'ecs, Hungary.}
\email{ilonasimon7@gmail.com}

\begin{abstract}
The aim of this paper is to establish density properties in $L^p$ spaces of the span of powers of functions
$\{\psi^\lambda\,:\lambda\in\Lambda\}$, $\Lambda\subset\N$ in the spirit of the M\"untz-Sz\'asz Theorem.
As density is almost never achieved, we further investigate the density of powers and a modulation of powers
$\{\psi^\lambda,\psi^\lambda e^{i\alpha t}\,:\lambda\in\Lambda\}$. Finally, we 
establish a M\"untz-Sz\'asz Theorem for density of translates of powers of cosines
$\{\cos^\lambda(t-\theta_1),\cos^\lambda(t-\theta_2)\,:\lambda\in\Lambda\}$. Under some arithmetic restrictions
on $\theta_1-\theta_2$, we show that density is equivalent to a M\"untz-Sz\'asz condition on $\Lambda$
and we conjecture that those arithmetic restrictions are not needed.
Some links are also established with the recently introduced concept of 
Heisenberg Uniqueness Pairs.
\end{abstract}

\subjclass{41A10;42C15,65T99}

\keywords{M\"untz-Sz\'asz Theorem, Heisenberg uniqueness pairs}

\maketitle

\section{Introduction}

The aim of this paper is to establish density properties in $L^p$ spaces of the span of powers of a single or a pair
of functions in the spirit of the M\"untz-Sz\'asz Theorem.

Representing a generic function of some function space in terms of a family of simple functions is one of the
main tasks in analysis. For instance, complex analysis deals with functions that can be expressed as power series,
that is, the span of the functions $\{x^k,k\in\N\}$.
Fourier analysis deals with the representation of functions in terms
of the simple functions $\{\cos 2k\pi t,\sin 2k\pi t\}_{k\in\Z}$ or alternatively 
$\{e^{2 ik\pi t}=\bigl(e^{2 i\pi t}\bigr)^k\}_{k\in\Z}$. Exploring the spanning properties (basis, minimal set,...)
of the restricted trigonometric system
$\{e^{2 ik\pi t}\}_{k\in\Lambda}$, $\Lambda\subset\Z$ in various function spaces has lead to a considerable bulk of 
Literature ({\it see e.g.} \cite{Ru} as a starting point). In order to establish good spanning properties of
the restricted trigonometric system, the first step consists in knowing if this system is {\em total} (that is, if
its span is dense) in a given function space.
Our aim here is to set a basic stone for similar properties when the basic brick $e^{2i\pi t}$ is replaced
by some other functions. 

When considering the power functions $\{t^\lambda,\lambda\in\Lambda\}$, the problem dates back to the early 20th century.
This problem leads to one of the most intriguing results, the M\"untz-Sz\'asz Theorem \cite{Mu,Sz} which relates the density of 
powers $\{x^\lambda\,:\ \lambda\in\Lambda\}$ in $\cc([0,1])$ with an arithmetic property of $\Lambda$, 
namely the divergence of the series  $\dst\sum_{\lambda\in\Lambda\setminus\{0\}}\frac{1}{\lambda}$. 
This theorem has been extended in many ways, in particular to $L^p$ spaces, 
{\it see e.g.} \cite{BE,CE,Er,EJ,S} and the nice survey \cite{Al} for more on the subject. We will recall precise statements
needed here in the next section.

The question we are asking here is of the same nature but we want to allow powers of more general functions than
the identity. More precisely, we want to investigate the density of systems of the form $\{\psi^\lambda\,: \ \lambda\in\Lambda\}$
in $L^p([a,b])$ or $\cc([a,b])$ when $\psi\,:[a,b]\to\R$ is a smooth function and $\Lambda$ is a set of integers ($\psi$
may change sign). It is rather easy to notice that such a density can only occur when $\psi$ is monotonic ({\it see}
Proposition \ref{prop:onefunction} below). On the other hand, if $\psi$ has a local extrema then $\psi$
has some symmetry and this symmetry will also occur in the entire closed span of $\{\psi^\lambda\,: \ \lambda\in\Lambda\}$.
Therefore, density can not be achieved for such functions. The question then arises on how to complete
this system in order to obtain density.

One idea is to add translations of $\psi$. For instance, for a given $f\in\cc([0,1])$ (here seen as the space of $1$-periodic functions),
we can consider the space
$$
\mathcal{T}(f)=\vect\{f^n(t-\tau),\ n\in\N,\tau\in[0,1]\}.
$$
As $\tt(\cos 2\pi t)$ is an algebra under pointwise multiplication, then, according to the Stone-Weierstrass Theorem, it is dense in
$\cc([0,1])$. This has been further investigated by Kerman and Weit \cite{KW} who gave a characterization of
the $f$'s for which $\tt(f)$ is dense in $\cc([0,1])$. Further generalizations can be found {\it e.g.} in \cite{RSW}.
We address here a similar question for $f(t)=\cos2\pi t$ and we show that the set of powers and translates can then be substantially
reduced. This should call for more research on density of 
$$
\mathcal{T}_{\Lambda,T}(f)=\vect\{f^\lambda(t-\tau),\ \lambda\in\Lambda,\ \tau\in T\}.
$$

A second option consists in adding modulations, instead of translate. In other words we are now
looking for density criteria for
$$
\mm_{\Lambda,\Omega}(f)=\vect\{f^\lambda(t)e^{i\omega t},\ \lambda\in\Lambda,\ \omega\in \Omega\}.
$$
Here we show that if $\Lambda$ satisfies a M\"untz-Sz\'asz type condition, two modulations suffice when $f$ has a single local maximum.

More precisely, our main results can be stated as follows (the general statement is more precise):

\medskip

\noindent{\bf Theorem.} {\sl Let $\Lambda$ be a set of non-negative integers containing zero and write 
$\Lambda=\{0\}\cup \Lambda_e\cup\Lambda_o$ where $\Lambda_e$ (resp. $\Lambda_o$) are the non-zero even (resp. odd)
integers in $\Lambda$. Let $\theta_1,\theta_2\in\R$ such that $\theta_1-\theta_2$ is an irrational algebraic number
and let $T=\{\theta_1,\theta_2\}$ and $\Omega=\{0,\omega\}$ with $|\omega|<1/2$.
Then the following are equivalent:
\begin{enumerate}
\item $\dst\sum_{\lambda\in\Lambda_e}\frac{1}{\lambda}=+\infty$ and $\dst\sum_{\lambda\in\Lambda_o}\frac{1}{\lambda}=+\infty$;

\item $\mm_{\Lambda,\Omega}\bigl(\cos\pi t\bigr)$ is dense in $L^p([0,1])$, $1<p<+\infty$;

\item $\tt_{\Lambda,T}(\cos 2\pi t)$ is dense in $L^p([0,1])$, $1<p<+\infty$.
\end{enumerate}
Moreover, the result stays true if $L^p([0,1])$ is replaced by $\cc([0,1])$.

For $\mm_{\Lambda,\Omega}$, the function $\cos\pi t$ can be replaced by any $\cc^2$ smooth function $\psi\,:[0,1]\to\R$ such that 
$\psi'$ vanishes at a single point $t_0\in(0,1)$ and $\psi''(t_0)\not=0$.

}

\medskip

We conjecture that the density of $\tt_{\Lambda,T}(\cos 2\pi t)$ is valid as soon as $\theta_1-\theta_2$ is irrational,
while we prove that it is not valid when $\theta_1-\theta_2$ is rational.

\medskip

The remaining of the paper is organized as follows. In the next section, we present some background on the M\"untz-Sz\'asz Theorem.
We then devote a section to our results on modulations while in Section 4 we prove our result concerning density of translates 
of the cosine function. In the last section we conclude by establishing some links with the recently introduced concept of Heisenberg
Uniqueness Pairs.

\section{Background and notations}

\begin{definition}
Let $\Lambda\subset\N:=\{0,1,2,\ldots\}$ and $I=[a,b]$, $a<b$ be a bounded interval. 
We will denote by $\Lambda_e=\Lambda\cap(2\N\setminus\{0\})$ and $\Lambda_o=\{0\}\cup\bigl(\Lambda\cap(2\N+1)\bigr)$.

Let us define an ($I$-MS) sequence in the following way:
\begin{itemize}
\item when either $a=0$ or $b=0$, then we call $\Lambda$ an $I$-M\"untz-Sz\'asz sequence, if $0\in\Lambda$ and
$$
\sum_{\lambda\in\Lambda\setminus\{0\}}\frac{1}{\lambda}=+\infty;
$$

\item when either $a>0$ or $b<0$, then we call $\Lambda$ an $I$-M\"untz-Sz\'asz sequence, if
$$
\sum_{\lambda\in\Lambda\setminus\{0\}}\frac{1}{\lambda}=+\infty;
$$
\item when $a<0<b$, then we call $\Lambda$ an $I$-M\"untz-Sz\'asz 
sequence, if $0\in\Lambda$,
$$
\sum_{\lambda\in\Lambda_e}\frac{1}{\lambda}=+\infty
\quad\mbox{and}\quad
\sum_{\lambda\in\Lambda_o}\frac{1}{\lambda}=+\infty.
$$
\end{itemize}
\end{definition}

We will further use the following notation: for $p\in[1,\infty]$, we write $X_p(I)=L^p(I)$ if $1\leq p<+\infty$
and $X_\infty(I)=\cc(I)$. We then define $p'$ to be the usual dual index, $\dst\frac{1}{p}+\frac{1}{p'}=1$ with the convention that
$1/\infty=0$. Finally, we write $X_p^{\prime}=X_{p'}$.

The classical M\"untz-Sz\'asz Theorem \cite[page 23]{S}, {\it see also} \cite[Section 6]{BE}, states that

\begin{theorem}[M\"untz-Sz\'asz]
Let $\Lambda\subset\N$, $1\leq p\leq+\infty$, $\dst\frac{1}{p}+\frac{1}{p'}=1$ and $I\subset\R$ be a bounded interval. The following conditions are equivalent
\begin{enumerate}
\renewcommand{\theenumi}{\roman{enumi}}
\item\label{MZ1} The set $\{x^\lambda,\lambda\in\Lambda\}$ is total in $X_p(I)$.

\item\label{MZ3} If $f\in X_p^\prime(I)$ is such that $\dst\int_I f(s)s^\lambda\,\mbox{d}s=0$
for every $\lambda\in\Lambda$, then $f=0$.

\item\label{MZ2} $\Lambda$ is an $I$-M\"untz-Sz\'asz sequence.
\end{enumerate}

Moreover, 

--- if $I\subset\R^+$ or $\R^-$ and $\sum_{\lambda\in\Lambda\setminus\{0\}}\frac{1}{\lambda}<+\infty$, then
every function in the closed linear span of $\{x^\lambda,\lambda\in\Lambda\}$ is analytic in the interior of $I$;

--- if $I=[a,b]$ with $a<0<b$ and $\sum_{\lambda\in\Lambda_e}\frac{1}{\lambda}<+\infty$
(resp. $\sum_{\lambda\in\Lambda_o}\frac{1}{\lambda}<+\infty$) then the even (resp. odd)
part of each function in the closed linear span of $\{x^\lambda,\lambda\in\Lambda\}$ is analytic
on $(a,b)\setminus\{0\}$.
\end{theorem}

Of course, the equivalence of \eqref{MZ1} and \eqref{MZ3} is a direct consequence of the Hahn-Banach Theorem.
The classical M\"untz-Sz\'asz Theorem covers only the case $I=[0,1]$ (and thus $I=[a,b]$ with $ab=0$), 
the more general case $I=[a,b]$, $ab\not=0$ is due to Clarkson-Erd\H{o}s and Schwartz. 
The case where $I$ is no longer included in a half-line is an easy consequence of the classical M\"untz-Sz\'asz Theorem by writing 
$f$ in \eqref{MZ3} as a sum of an even and odd function (after extending $f$ by $0$
so that it is defined on a symmetric interval). Also, this theorem is usually stated for density in $\cc(I)$
but the statement is the same for $L^p(I)$ when $\Lambda\subset\N$, {\it see e.g.} \cite[Section 6]{BE}.

Note that when $I$ intersects both $\R_+$ and $\R_-$ the statement can be reformulated in terms of the Fourier transform
that we normalize as
$$
\widehat{f}(\zeta)=\ff[f](\zeta):=\int_\R f(s)e^{-is\zeta}\,\mbox{d}s
$$
if $f\in L^1(\R)$ and extended to $L^2(\R)$ in the usual way.
In this case, if $\zeta_0\in\R$, then \eqref{MZ1},\eqref{MZ3},\eqref{MZ2} are equivalent to

\smallskip
\begin{enumerate}
\renewcommand{\theenumi}{\roman{enumi}}
\setcounter{enumi}{3}
\item {\sl $\dst\frac{\mathrm{d}^\lambda}{\mathrm{d}x^\lambda}\widehat{f}(\zeta_0)=0$, for every $\lambda\in\Lambda$
implies $f=0$.}
\end{enumerate}

\section{Density of powers of a fixed function and modulation}

In this section $I$ will still be a fixed bounded closed interval and $\psi\,:I\to\R$ a $\cc^1$-smooth function
(one may slightly weaken this condition). We will first prove the following result:

\begin{proposition}
\label{prop:onefunction}
Let $a,b\in\R$ and $\psi\,:[a,b]\to\R$ be a $\cc^2$ function such that $\psi'$ and $\psi''$ do not vanish 
simultaneously. Let $p\in[1,+\infty]$. Let $J=\psi([a,b])$ and let $\Lambda\subset\N$.
The following are equivalent:
\begin{enumerate}
\renewcommand{\theenumi}{\roman{enumi}}
\item $\{\psi^\lambda\,:\lambda\in\Lambda\}$ is total in $X_p(a,b)$.

\item $\psi$ is one-to-one and $\Lambda$ is a $J$-M\"untz-Sz\'asz sequence.
\end{enumerate}
\end{proposition}

\begin{proof}
Let us first assume that $\psi$ is {\em not} one-to-one. Then $\psi$ has a local extremum at a point $x_0$
in the interior of $[a,b]$. Therefore, there exists $a\leq a'<x_0<b'\leq b$ and a map $\ffi\,:[a',x_0]\to [x_0,b']$
such that $\ffi$ is one-to-one and onto and $\psi\circ\ffi=\psi$ on $[a',x_0]$. Let $f$ be any non-zero $\cc^1$ function
on $[x_0,b']$ and extend $f$ to $[a',x_0]$ by setting
$$
f(x)=-\ffi'(x)f\bigl(\ffi(x)\bigr)
$$
and then extend $f$ further to $[a,b]\setminus[a',b']$ by setting $f(x)=0$. Then $f\in L^{p'}(a,b)$ ($1/p+1/p'=1$) is non-zero and
$$
\int_a^bf(x)\psi^\lambda(x)\,\mbox{d}x=\int_{a'}^{b'}f(x)\psi^\lambda(x)\,\mbox{d}x
=\int_{a'}^{x_0}+\int_{x_0}^{b'}f(x)\psi^\lambda(x)\,\mbox{d}x.
$$
Changing variable $x=\ffi(t)$ in the first integral we obtain
$$
\int_a^bf(x)\psi^\lambda(x)\,\mbox{d}x=
\int_{b'}^{x_0}f\bigl(\ffi(x)\bigr)\psi^\lambda\bigl(\ffi(x)\bigr)\ffi'(x)\,\mbox{d}x
+\int_{x_0}^{b'}f(x)\psi^\lambda(x)\,\mbox{d}x=0.
$$
It follows that $\{\psi^\lambda\,:\lambda\in\Lambda\}$ is {\em not} total in $X_p(a,b)$.

Let us now assume that $\psi$ is one-to-one so that $\psi'$ does not vanish (otherwise, if $\psi'(x_0)=0$
then, by assumption, $\psi''(x_0)\not=0$ so that $\psi'$ changes sign at $x_0$ and $\psi$ would not be one-to-one). 
In particular, $|\psi'|$ is bounded below. For a function $f$ on $(a,b)$ we define the function $\psi_*f$ on $J$
by $\dst \psi_*f(t)=\frac{f\bigl(\psi^{-1}(t)\bigr)}{\psi'\bigl(\psi^{-1}(t)\bigr)}$. Then, as $|\psi'|$ is bounded below,
$f\in X_p^\prime(a,b)$ if and only if $\psi_*f\in X_p^\prime(J)$.

Further, changing variable $t=\psi(x)$ we get
$$
\int_a^bf(x)\psi^\lambda(x)\,\mbox{d}x=\int_J\frac{f\bigl(\psi^{-1}(t)\bigr)}{\psi'\bigl(\psi^{-1}(t)\bigr)} t^\lambda\,\mbox{d}t.
$$
Applying
the M\"untz-Sz\'asz Theorem, the above proposition follows.
\end{proof}

The question now arises on how to modify the set $\{\psi^\lambda,\lambda\in\Lambda\}$ in order to obtain a total set when $\psi$
is not one-to-one. In our opinion, there are two natural ways to do so, if one considers the M\"untz-Sz\'asz theorem as a statement about
the cancellation of the Fourier transform of a compactly supported function in a point.
The first one consists of adding modulations the second one consists of adding translations. The following result deals with modulations
and shows the equivalence $(1)\Leftrightarrow(2)$ of the theorem stated in the introduction.

\begin{theorem}
Let $a,b\in\R$, $\psi$ be a $\cc^2$ function $[a,b]\to\R$ such that
$\psi'$ changes sign in a single point $x_0\in(a,b)$.
Let $\dst-\frac{1}{b-a}<\alpha<\frac{1}{b-a}$ and define $e_\alpha(t)=e^{i\alpha t}$.
Let $\Lambda,\Lambda'\subset\N$ and $p\in(1,+\infty]$. The following are equivalent.
\begin{enumerate}
\renewcommand{\theenumi}{\roman{enumi}}
\item\label{hyp1} $\{\psi^\lambda,\lambda\in\Lambda\}\cup\{\psi^\lambda e_\alpha,\lambda\in\Lambda'\}$ is total in $X_p$.

\item\label{hyp2} $\Lambda$ and $\Lambda'$ are both $\psi([a,b])$-M\"untz-Sz\'asz sequences. 
\end{enumerate}
\end{theorem}

\begin{example}
Typical examples we have in mind are the functions $\psi(t)=\cos \pi t$ and $\psi(t)=1-\cos \pi t$ on $[-1/2,1/2]$ or equivalently $\psi(t)=\sin\pi t$ and $\psi(v)=1-\sin\pi t$ on $[0,1]$.

Further examples are $\psi(t)=t^2$ on $[-1,1]$, $\psi(t)=1-t^2$ on $[0,1]$. A translation and dilation
then gives a density criteria for the family $\{[t(1-t)]^\lambda,[t(1-t)]^\lambda e^{it}\,:\lambda\in\Lambda\}$ in $L^p([0,1])$.
\end{example}

\begin{remark}
The function $e_\alpha$ can be replaced by any function of the form $e^{i\ffi(t)}$ where the real valued function $\ffi$ is chosen 
such that, if $\psi(v)=\psi(v')$ with $v\not=v'$, then $e^{i\bigl(\ffi(v)-\ffi(v')\bigr)}\not=1$.

If one chooses $\ffi$ such that, if $\psi(v)=\psi(v')$ with $v\not=v'$ implies $e^{i\bigl(\ffi(v)-\ffi(v')\bigr)}\not=\pm 1$,
then the same result stays true for the system $\{\psi^\lambda\cos\ffi,\lambda\in\Lambda\}\cup\{\psi^\lambda\sin\ffi,\lambda\in\Lambda'\}$.
We leave the necessary adaptation of the proof below to the reader.
\end{remark}

\begin{proof}  Let $p'$ be given by $\dst\frac{1}{p}+\frac{1}{p'}=1$. 
We will only prove the theorem for $p\in(1,\infty)$ as no change is needed for $p=+\infty$. 

We will use the following notation: set $J=\psi([a,b])$, $J_+=\psi([x_0,b])$ and $J_-=\psi([a,x_0])$
and $\psi_+^{-1}\,:J_+\to[x_0,b]$ be the inverse of $\psi$ on $[x_0,b]$
while $\psi_-^{-1}\,:J_-\to[x_0,b]$ is the inverse of $\psi$ on $[a,x_0]$.
Then
\begin{eqnarray*}
\int_a^b f(x)\psi(x)^\lambda\,\mbox{d}x&=&
\int_a^{x_0} f(x)\psi(x)^\lambda\,\mbox{d}x+\int_{x_0}^b f(x)\psi(x)^\lambda\,\mbox{d}x\\
&=&\int_{J_-}\frac{f\bigl(\psi_-^{-1}(y)\bigr)}{\psi'\bigl(\psi_-^{-1}(y)\bigr)}y^\lambda\,\mbox{d}y
+\int_{J_+}\frac{f\bigl(\psi_+^{-1}(y)\bigr)}{\psi'\bigl(\psi_+^{-1}(y)\bigr)}y^\lambda\,\mbox{d}y\\
&=&\int_J\left(
\mathbf{1}_{J_-}(y)\frac{f\bigl(\psi_-^{-1}(y)\bigr)}{\psi'\bigl(\psi_-^{-1}(y)\bigr)}
+\mathbf{1}_{J_+}(y)\frac{f\bigl(\psi_+^{-1}(y)\bigr)}{\psi'\bigl(\psi_+^{-1}(y)\bigr)}\right)y^\lambda\,\mbox{d}y.
\end{eqnarray*}
But then if we set\footnote{with the obvious abuse of notation when $y\notin J_-$ or $y\notin J_+$.} 
$$
g(y)=\mathbf{1}_{J_-}(y)\frac{f\bigl(\psi_-^{-1}(y)\bigr)}{\psi'\bigl(\psi_-^{-1}(y)\bigr)}
+\mathbf{1}_{J_+}(y)\frac{f\bigl(\psi_+^{-1}(y)\bigr)}{\psi'\bigl(\psi_+^{-1}(y)\bigr)}
$$
we get
$$
\int_a^b f(x)\psi(x)^\lambda\,\mbox{d}x=\int_J g(y)y^\lambda\,\mbox{d}y.
$$
Similarly, if we set 
$$
\tilde g(x)=\mathbf{1}_{J_-}(y)\frac{f\bigl(\psi_-^{-1}(y)\bigr)e^{i\alpha \psi_-^{-1}(y)}}{\psi'\bigl(\psi_-^{-1}(y)\bigr)}
+\mathbf{1}_{J_+}(y)\frac{f\bigl(\psi_+^{-1}(y)\bigr)e^{i\alpha \psi_+^{-1}(y)}}{\psi'\bigl(\psi_+^{-1}(y)\bigr)}
$$
we get
$$
\int_a^b f(x)e^{i\alpha x}\psi(x)^\lambda\,\mbox{d}x=\int_J \tilde g(y)y^\lambda\,\mbox{d}y.
$$

Let us now prove \eqref{hyp2}$\Rightarrow$\eqref{hyp1}. Assume that $0$ is not in the interior of $J$ and that $\Lambda$ and $\Lambda'$ 
are both $J$-M\"untz-Sz\'asz sequences (the proof when $0$ is in the interior of $J$
and $\Lambda_e,\Lambda_o$ and $\Lambda^\prime_e,\Lambda^\prime_o$ are all
M\"untz-Sz\'asz sequences is similar). Notice that if $f\in L^p(a,b)$, then $g,\tilde g\in L^1(J)$.
According to the M\"untz-Sz\'asz Theorem, if
$$
\int_a^b f(x)\psi(x)^\lambda\,\mbox{d}x=\int_a^b f(x)\psi(x)^{\lambda'}e^{i\alpha x}\,\mbox{d}x=0
$$
for every $\lambda\in\Lambda,\lambda'\in\Lambda'$, then $g=\tilde g=0$. But, writing
$$
f_-(y)=\mathbf{1}_{J_-}(y)\frac{f\bigl(\psi_-^{-1}(y)\bigr)}{\psi'\bigl(\psi_-^{-1}(y)\bigr)}
\quad\mbox{and}\quad
f_+(y)=\mathbf{1}_{J_+}(y)\frac{f\bigl(\psi_+^{-1}(y)\bigr)}{\psi'\bigl(\psi_+^{-1}(y)\bigr)}
$$
and $u_\pm=e^{i\alpha\psi_\pm^{-1}(y)}$,
$g=\tilde g=0$ is equivalent to
$$
\left\{\begin{matrix} f_-(y)&+&f_+(y)&=&0\\
u_-f_-(y)&+&u_+f_+(y)&=&0
\end{matrix}\right..
$$
As $u_-\not=u_+$, this implies $f_+=f_-=0$ thus $f=0$.

Conversely, for \eqref{hyp1}$\Rightarrow$\eqref{hyp2}, assume that one of $\Lambda,\Lambda'$ is not a $J$-M\"untz-Sz\'asz sequence. Let $\tilde p=\frac{3p'}{3p'-1}$ so that
$\dst\frac{1}{\tilde p}+\frac{1}{3p'}=1$. Applying the M\"untz-Sz\'asz Theorem in $L^{\tilde p}(J)$,
there exist $g,\tilde g\in L^{3p'}$, one of them non zero
and the other $0$, such that
$$
\int_J g(y)y^\lambda\,\mbox{d}y=\int_J \tilde g(y)y^{\lambda'}\,\mbox{d}y=0
$$
for every $\lambda\in\Lambda,\lambda'\in\Lambda'$. If we find $f\in L^{p'}([a,b])$ such that
the associated $f_\pm$ are solution of 
\begin{equation}
\label{eq:syst}
\left\{\begin{matrix} f_-(y)&+&f_+(y)&=&g\\
u_-f_-(y)&+&u_+f_+(y)&=&\tilde g
\end{matrix}\right.
\end{equation}
then also
$$
\int_a^b f(x)\psi(x)^\lambda\,\mbox{d}x=\int_a^b f(x)\psi(x)^{\lambda'}e^{i\alpha x}\,\mbox{d}x=0.
$$
But, the system \eqref{eq:syst} has as solution
$$
f_+(y)=\frac{u_-g-\tilde g}{u_--u_+}
\quad\mbox{and}\quad
f_-(y)=-\frac{u_+g-\tilde g}{u_--u_+}
$$
that is
$$
\mathbf{1}_{J_+}(y)f\bigl(\psi_+^{-1}(y)\bigr)=\psi'\bigl(\psi_+^{-1}(y)\bigr)
\frac{e^{i\alpha\psi_-^{-1}(y)}g(y)-\tilde g(y)}
{e^{i\alpha\psi_-^{-1}(y)}-e^{i\alpha\psi_+^{-1}(y)}}
$$
and
$$
\mathbf{1}_{J_-}(y)f\bigl(\psi_-^{-1}(y)\bigr)=-\psi'\bigl(\psi_-^{-1}(y)\bigr)
\frac{e^{i\alpha\psi_+^{-1}(y)}g(y)-\tilde g(y)}
{e^{i\alpha\psi_-^{-1}(y)}-e^{i\alpha\psi_+^{-1}(y)}}.
$$
From this, we get
$$
f(x)=\begin{cases}\psi'(x)\frac{e^{i\alpha\psi_+^{-1}\circ\psi(x)}g\circ\psi(x)-\tilde g\circ\psi(x)}
{e^{i\alpha\psi_+^{-1}\circ\psi(x)}-e^{i\alpha x}}
&\mbox{for }x\in[a,x_0]\\
\psi'(x)\frac{e^{i\alpha\psi_-^{-1}\circ\psi(x)}g\circ\psi(x)-\tilde g\circ\psi(x)}
{e^{i\alpha\psi_-^{-1}\circ\psi(x)}-e^{i\alpha x}}
&\mbox{for }x\in[x_0,b]
\end{cases}.
$$
But now, as exactly one of $g,\tilde g$ is zero, $f$ is not the zero function. It remains to prove that
$f\in L^{p'}(a,b)$. For this, define $f_-$ on $[a,x_0]$ and $f_+$ on $[x_0,b]$ by
$$
f_\pm(x)=\psi'(x)\frac{e^{i\alpha\psi_\mp^{-1}\circ\psi(x)}g\circ\psi(x)}
{e^{i\alpha\psi_\mp^{-1}\circ\psi(x)}-e^{i\alpha x}}
$$
and $\tilde f_\pm=f-f_\pm$. It is enough to show that
$f_-,\tilde f_-\in L^{p'}(a,x_0)$ and $f_+,\tilde f_+\in L^{p'}(x_0,b)$.

Next, changing variable $x=\psi_-^{-1}(t)$, we get
\begin{eqnarray*}
\int_a^{x_0}|f_-(x)|^{p'}\,\mbox{d}x
&=&\int_a^{x_0}\abs{\psi'(x)\frac{e^{i\alpha\psi_+^{-1}\circ\psi(x)}g\circ\psi(x)}
{e^{i\alpha\psi_+^{-1}\circ\psi(x)}-e^{i\alpha x}}}^{p'}\,\mbox{d}x\\
&=&\int_{J_-}\
\frac{|\psi'\bigl(\psi_-^{-1}(t)\bigr)|^{p'-1}}{\abs{e^{i\alpha\psi_+^{-1}(t)}-e^{i\alpha\psi_-^{-1}(t)}}^{p'}}|g(t)|^{p'}
\,\mbox{d}t\\
&=&\frac{1}{2^{p'}}\int_{J_-}
\frac{|\psi'\bigl(\psi_-^{-1}(t)\bigr)|^{p'-1}}{\abs{\sin \frac{\alpha}{2}\bigl(\psi_+^{-1}(t)-\psi_-^{-1}(t)\bigr)}^{p'}}
|g(t)|^{p'}\,\mbox{d}t.
\end{eqnarray*}

But now, as $\psi''(x_0)\not=0$, $\psi(x)=\psi(x_0)+\frac{\psi''(x_0)}{2}(x-x_0)^2+o\bigl((x-x_0)^2\bigr)$. From this, one immediately gets that
$$
\Phi(t)=\frac{|\psi'\bigl(\psi_-^{-1}(t)\bigr)|^{p'-1}}{\abs{\sin \frac{\alpha}{2}\bigl(\psi_+^{-1}(t)-\psi_-^{-1}(t)\bigr)}^{p'}}
\approx C\bigl(t-\psi(x_0)\bigr)^{-1/2}
$$
when $t\to \psi(x_0)$ one of the end points of $J_-$. Further, the assumption on $\psi$ implies that
$\Phi$ is $\cc^2$ smooth on $J_-\setminus\{\psi(x_0)\}$. In particular, $\Phi\in L^{3/2}(J_-)$ (say).
Thus, from H\"older's inequality,
$$
\int_a^{x_0}|f_-(x)|^{p'}\,\mbox{d}x\leq \frac{1}{2^{p'}}\norm{\Phi}_{L^{3/2}(J_-)}\norm{g}_{L^{3p'}(J_-)}^{p'}<+\infty.
$$
The proof for $f_+$ and $\tilde f_\pm$ is similar.
\end{proof}

\section{Density of translates of powers of the cosine function}

In this section, functions on $[0,1]$ will be identified with $1$-periodic functions, so that even and odd functions
on $[0,1]$ make sense.
We are interested in the density of translates of powers of the cosine (or sine) function,
$\{\cos^\lambda 2\pi(t-\theta),\lambda\in\Lambda\}$. According to Proposition \ref{prop:onefunction},
this system is never dense in $L^p(0,1)$. Actually, for this function, it is easy to describe the ``orthogonal'':

\begin{lemma}\label{lem:oneshift}
Let $p\in [1,+\infty]$, $\Lambda\subset\N$, $\theta\in[0,1)$ and
$$
\tt_{p,\Lambda,\theta}=\overline{\vect}\{\cos^\lambda 2\pi(t-\theta),\lambda\in\Lambda\}
$$
be the closed subspace of $X_p(0,1)$ spanned by $\{\cos^\lambda 2\pi(t-\theta),\lambda\in\Lambda\}$.
Let 
$$
\tt_{p,\Lambda,\theta}^\perp=\left\{f\in X_p^\prime(0,1)\,: \int_0^1 f(t)\cos^\lambda 2\pi(t-\theta)\,\mbox{d}t=0\quad\forall\lambda\in\Lambda\right\}. 
$$

If $\Lambda$ is a $[-1,1]$-M\"untz-Sz\'asz sequence, then 
$$
\tt_{p,\Lambda,\theta}=\{f\in X_p(0,1)\,: f(\theta+t)-f(\theta-t)=0\mbox{ a.e. on }[0,1]\}
$$
and
$$
\tt_{p,\Lambda,\theta}^\perp=\{f\in X_p^\prime(0,1)\,:f(\theta+t)+f(\theta-t)=0\mbox{ a.e. on }[0,1]\}.
$$
\end{lemma}

\begin{proof}[Proof of Lemma \ref{lem:oneshift}] Up to translating by $\theta$, we may assume that $\theta=0$.
We thus want to prove that $\tt_{p,\Lambda,\theta}^\perp$ is the space of odd functions in $X_p^\prime(0,1)$. Once this is established,
it is obvious that $\tt_{p,\Lambda,\theta}$ is the space of even functions in $X_p(0,1)$.

Note that
$$
\int_0^1f(t)\cos^\ell 2\pi t\,\mbox{d}t=\int_{-1/2}^{1/2}f(t)\cos^\ell 2\pi t\,\mbox{d}t
=\int_0^{1/2}\bigl(f(t)+f(-t)\bigr)\cos^\ell 2\pi t\,\mbox{d}t.
$$
We can now change variable $u=\cos 2\pi t$ to get
\begin{equation}
\label{eq:fgt}
\int_0^1f(t)\cos^\ell 2\pi(t)\,\mbox{d}t=\int_{-1}^1g(u)u^\ell\,\mbox{d}u
\end{equation}
where
\begin{equation}
\label{defgtheta}
g(u)=\frac{1}{2\pi\sqrt{1-u^2}}\ent{f\left(\frac{\arccos u}{2\pi}\right)+f\left(-\frac{\arccos u}{2\pi}\right)}.
\end{equation}
Now $g\in L^1(-1,1)$ since
\begin{eqnarray*}
\int_{-1}^1|g(u)|\,\mbox{d}u&\leq&
\frac{1}{2\pi}\int_{-1}^1\abs{f\left(\frac{\arccos u}{2\pi}\right)}+\abs{f\left(-\frac{\arccos u}{2\pi}\right)}
\frac{\,\mbox{d}u}{(1-u^2)^{1/2}}\\
&=&\int_0^{1/2}\bigl(|f(t)|+|f(-t)|\bigr)\,\mbox{d}t<+\infty
\end{eqnarray*}
as $f\in X_p(0,1)\subset L^1(0,1)$. If $\Lambda$ is a $[-1,1]$-M\"untz-Sz\'asz sequence
we deduce that $g=0$ which is equivalent to $f(t)+f(-t)=0$ {\it i.e.} $f$ is odd.
\end{proof}

\begin{lemma}\label{lem:analytic}
Let $p\in [1,+\infty]$, $\Lambda\subset\N$, $\theta\in [0,1)$. Assume that $\Lambda$
is {\em not} a $[-1,1]$-M\"untz-Sz\'asz sequence. 
Let $f\in \tt_{p,\Lambda,\theta}$, that is, $f$ is even with respect to $\theta$,
$f(\theta+t)=f(\theta-t)$ a.e., and assume further that

--- if $\dst\sum_{\lambda\in\Lambda_e}\frac{1}{\lambda}=+\infty$, then $f(\theta+1/2-t)=f(\theta+t)$

--- if $\dst\sum_{\lambda\in\Lambda_o}\frac{1}{\lambda}=+\infty$, then $f(\theta+1/2-t)=-f(\theta+t)$.

Then $f$ is analytic on $[0,1)$ except possibly at two points.
\end{lemma}

\begin{remark}
Write $c_k(f)$ for the $k$-th Fourier coefficient of $f$ and note that
$f(\theta-t)=f(\theta+t)$ is equivalent to $c_{-k}(f)e^{-2i\pi k\theta}=c_k(f)e^{2i\pi k\theta}$ for all $k$.
On the other hand, $f(\theta+1/2-t)=\pm f(\theta+t)$ is equivalent to
$c_{-k}(f)e^{-2i\pi k\theta}=\pm (-1)^kc_k(f)e^{2i\pi k\theta}$ so that
$c_k=0$ when $k$ is odd --- resp. $c_k=0$ when $k$ is even.
\end{remark}

\begin{proof} Up to translating by $\theta$, we may assume that $\theta=0$ {\it i.e.}, $f$ is even.
Let $f\in\tt_{p,\Lambda,0}$ be even and define $h$ on $[-1,1]$ by
$$
h(x)=f\left(\frac{\arccos x}{2\pi}\right)
$$
so that, when $f(1/2-t)=f(t)$, $h$ is even, while
$f(1/2-t)=-f(t)$ implies that $h$ is odd.

Let $F\subset\Lambda_e$ when $h$ is even (resp. $F\subset\Lambda_o$ when $h$ is odd) be a finite set.

Changing variable $t=\frac{\arccos x}{2\pi}$, we get
\begin{eqnarray*}
\int_{-1/2}^{1/2}\abs{f(t)-\sum_{\lambda\in F}c_\lambda \cos^\lambda 2\pi t}^p\,\mbox{d}t&=&
2\int_0^{1/2}\abs{f(t)-\sum_{\lambda\in F}c_\lambda \cos^\lambda 2\pi t}^p\,\mbox{d}t\\
&=&\frac{1}{\pi}\int_{-1}^1\abs{h(x)-\sum_{\lambda\in F}c_\lambda x^\lambda}^p\,\frac{\mbox{d}x}{\sqrt{1-x^2}}\\
&=&\frac{2}{\pi}\int_{0}^1\abs{h(x)-\sum_{\lambda\in F}c_\lambda x^\lambda}^p\,\frac{\mbox{d}x}{\sqrt{1-x^2}}\\
&\geq&\frac{1}{\pi}\int_{0}^1\abs{ h(x)-\sum_{\lambda\in F}c_\lambda x^\lambda}^p\,\mbox{d}x.
\end{eqnarray*}
so that $h\in\pp_{p,\Lambda_e}$ (resp. $h\in\pp_{p,\Lambda_o}$) where we denote by
$$
\pp_{p,A}=\overline{\vect}\{t^a,a\in A\}
$$
the closed subspace of $X_p(-1,1)$ spanned by $\{x^a,a\in A\}$.

As $\Lambda_e$ (resp. $\Lambda_o$) does not satisfy the M\"untz-Sz\'asz condition, $h$ is
analytic on $(0,1)$ so that $h$ is analytic in $(-1,1)\setminus\{0\}$. As $f(t)=h(\cos 2\pi t)$, the result follows.
\end{proof}

As one system $\{\cos^\lambda 2\pi(t-\theta),\lambda\in\Lambda\}$ is not total in $L^p(0,1)$, we may ask if
adding a second system of this kind improves the situation.
Let us first show that the second system can not be arbitrary:

\begin{lemma}
Let $p\in[1,\infty]$,
$\theta_1,\theta_2\in[0,1)$ and $\Lambda,\Lambda'\subset\N$. Assume that the system 
$$
\{\cos^\lambda 2\pi(t-\theta_1),\lambda\in\Lambda\}\cup\{\cos^{\lambda'} 2\pi(t-\theta_2),\lambda'\in\Lambda'\}
$$
is total in $X_p(0,1)$, then $\theta_1-\theta_2\notin\Q$.
\end{lemma}

\begin{proof}
Let us write $\theta_1-\theta_2=\frac{m}{n}\in\Q$. Let $\ffi$ be a continuous, odd and $1/n$-periodic function
and $f(t)=\ffi(t-\theta_2)$. Then, for every $\ell\in\N$, as $f$ is also $1$-periodic,
$$
\int_0^1f(t)\cos^\ell 2\pi(t-\theta_2)\,\mbox{d}t=\int_{-1/2}^{1/2}\ffi(t)\cos^\ell 2\pi t\,\mbox{d}t=0
$$
since $\ffi$ is odd. Next, write $\theta_1=\dst\theta_2+\frac{m}{n}$ then, as $\ffi$ is $1/n$-periodic,
$$
f(t)=\ffi(t-\theta_2)=\ffi(t-\theta_2-m/n)=\ffi(t-\theta_1),
$$
therefore, for each $\ell\in\N$,
$$
\int_0^1f(t)\cos^\ell 2\pi(t-\theta_1)\,\mbox{d}t
=\int_0^1\ffi\left(t-\theta_1\right)\cos^\ell 2\pi \left(t-\theta_1\right)\,\mbox{d}t
=\int_{-1/2}^{1/2}\ffi(t)\cos^\ell 2\pi t\,\mbox{d}t=0,
$$
using $1$-periodicity again.
It follows that $\{\cos^\lambda 2\pi(t-\theta_1),\lambda\in\Lambda\}\cup\{\cos^{\lambda'} 2\pi(t-\theta_2),\lambda'\in\Lambda'\}$
is never total in $X_p(0,1)$.
\end{proof}

We will also need the following lemma which provides a (non-orthogonal) decomposition trigonometric polynomials
into a sum of two trigonometric polynomials with specific parity.
In a sense, this generalizes the decomposition of a function into an even and an odd function.
Unfortunately, it is only valid in full generality for trigonometric polynomials.

\begin{lemma}\label{lem:decomp}
Let $P$ be a $1$-periodic trigonometric polynomial with zero-mean and $\theta_1,\theta_2\in\R$ such that $\theta_1-\theta_2\notin\Q$. 
Then there exists a unique pair $(P_1,P_2)$ of $1$-periodic trigonometric polynomials with zero-mean
such that $P_1(\theta_1-t)=P_1(\theta_1+t)$, $P_2(\theta_2-t)=P_2(\theta_2+t)$ and $P=P_1+P_2$.
\end{lemma}

\begin{proof} By expanding $P,P_1,P_2$ in Fourier series, one sees that the existence of the desired decomposition 
is equivalent to the systems
$$
\left\{\begin{matrix}
c_k(P)&=&&c_k(P_1)&+&&c_k(P_2)\\
c_{-k}(P)&=&e^{4i\pi k\theta_1}&c_k(P_1)&+&e^{4i\pi k\theta_2}&c_k(P_2)
\end{matrix}\right.\quad,\ \forall k\in\Z\setminus\{0\}.
$$
As $\theta_1-\theta_2\notin\Q$,
 the determinant of this system is non zero for every $k\not=0$ and its solutions are given by
\begin{equation}
\label{solsyst}
c_k(P_1)=\frac{c_{-k}(P)-c_k(P)e^{4i\pi k\theta_2}}{e^{4i\pi k\theta_1}-e^{4i\pi k\theta_2}}
\quad,\quad
c_k(P_2)=\frac{c_{-k}(P)-c_k(P)e^{4i\pi k\theta_1}}{e^{4i\pi k\theta_2}-e^{4i\pi k\theta_1}}.
\end{equation}
which gives both existence and uniqueness.
\end{proof}

\begin{remark}
When $\theta_1-\theta_2=\frac{m}{n}\in \Q$, the lemma stays true with the same proof
if we impose $P,P_1,P_2$ to have degree
$<n/2$ for even $n$ and $<n$ for odd $n$.

Note also that the zero mean assumption is only used to guarantee uniqueness as constant functions satisfy
both parities $P_1(\theta_1-t)=P_1(\theta_1+t)$, $P_2(\theta_2-t)=P_2(\theta_2+t)$. Actually, the proof
shows that the $P_1,P_2$'s we obtain have both zero mean. If $P$ has non-zero mean, we apply the lemma
to $P-c_0(P)$, which has zero mean. We may thus write
$P=c_0(P)+P_1+P_2=\bigl( P_1+\lambda c_0(P)\bigr)+\bigl( P_2+(1-\lambda)c_0(P)\bigr)$, $\lambda\in\R$.
Uniqueness would still be guaranteed if we ask for, say, $P_2$ to have zero mean which would imply that we take $\lambda=1$ in
the above decomposition.
\end {remark}

Note that if $P$ is no longer a trigonometric polynomial but a function in $L^1$, the formula \eqref{solsyst}
will in general lead to sequences that do not go to zero, so that they are not sequences of Fourier coefficients of $L^1$
functions.

Before going on with our main subject, let us elaborate a bit on this topic:

\begin{lemma}\label{lem:decom}
 Let $\theta_1,\theta_2\in\R$ such that $\theta:=\theta_1-\theta_2\in\R\setminus\Q$. Let $a>0$ and assume that
$\theta$ is $a$-approximable by rational numbers in the sense that there is a constant $C_\theta>0$ such that the set
$$
\{(m,n)\in\Z\,:\ |m-n\theta|< C_\theta n^{-a}\}
$$
is finite.

Let $s\geq a$ and
$f\in H^s(0,1)$ with mean zero, then there exists a unique pair $f_1,f_2\in L^2(0,1)$ such that $f=f_1+f_2$,
$f_1(\theta_1-t)=f_1(\theta_1+t)$ and $f_2(\theta_2-t)=f_2(\theta_2+t)$.

Moreover, if $s>a+1/2+j$ for some integer $j$, then $f_1,f_2$ are of class $\cc^j$.
\end{lemma}

\begin{remark} From the Dirichlet Theorem, no irrational number is $1$-approximable. However,
according to Khinchin's theorem that for $a>1$, almost every real number is $a$-approximable.
Further, from the Thue-Siegel-Roth Theorem, every algebraic number is $1+\eps$-approximable for every $\eps>0$.
On the other hand, Liouville numbers are not $a$-approximable for any number. {\it See e.g.} \cite{HS,Sc,Wa} and references therein for
more on the subject.
\end{remark}

\begin{proof} As previously, if the decomposition exists then the Fourier coefficients of $f_1,f_2$
are given by
$$
c_k(f_1)=\frac{c_{-k}(f)-c_k(f)e^{4i\pi k\theta_2}}{e^{4i\pi k\theta_1}-e^{4i\pi k\theta_2}}
\quad,\quad
c_k(f_2)=\frac{c_{-k}(f)-c_k(f)e^{4i\pi k\theta_1}}{e^{4i\pi k\theta_2}-e^{4i\pi k\theta_1}}.
$$
Conversely, if these two sequences are in $\ell^2$ (resp.  if $k^jc_k(f_1),k^jc_k(f_2)\in\ell^1$) then the corresponding
Fourier series define $f_1,f_2$ as functions in $L^2(0,1)$ (resp. in $\cc^j(0,1)$).

Since $f\in H^s$, $(1+|k|)^sc_k(f)\in\ell^2$. On the other hand
$$
|e^{4i\pi k\theta_1}-e^{4i\pi k\theta_2}|=2|\sin \pi k 2\theta|=2|\sin \pi\dist (2k\theta,\Z)|
\geq 4\dist (2k\theta,\Z).
$$
But,
$\dist (2k\theta,\Z)>2^{-a}C_\theta k^{-a}$ for all but finitely many $k$'s so that 
$|e^{4i\pi k\theta_1}-e^{4i\pi k\theta_2}|>2^{2-a}C_\theta |k|^{-a}$ for all but finitely many $k$'s,
thus for all $|k|\geq k_\theta$ for some $k_\theta\in\N$.

Now, if $f\in H^s$, for $|k|\geq k_\theta$,
$$
\frac{|c_{k}(f)|}{|e^{4i\pi k\theta_1}-e^{4i\pi k\theta_2}|}\leq \frac{2^{a-2}}{C_\theta} |k|^a |c_{k}(f)|\in \ell^2(\Z)
$$
if $s\geq a$.
Further
\begin{eqnarray*}
\sum_{|k|\geq k_\theta}\frac{|k|^j|c_{k}(f)|}{|e^{4i\pi k\theta_1}-e^{4i\pi k\theta_2}|}
&\leq&\sum_{|k|\geq k_\theta}\frac{|k|^s|c_{k}(f)|}{2^{2-a}C_\theta |k|^{s-j-a}}\\
&\leq&\frac{2^{a-2}}{C_\theta}
\left(\sum_{|k|\geq k_\theta}\frac{1}{|k|^{2(s-j-a)}}\right)^{1/2}\left(\sum_{|k|\geq k_\theta}|k|^{2s}|c_{k}(f)|^2\right)^{1/2}
\end{eqnarray*}
which is finite as soon as $s>j+a+1/2$. The result follows immediately.
\end{proof}

\begin{remark} 
Note that if $f$ is such that $c_k(f)=0$ when $k$ is odd --- resp. $c_k(f)=0$ when $k$ is even---
then the same is true for $f_1,f_2$. Thus, for $j=1,2$,
$f_j(\theta+1/2-t)=f_j(\theta+t)$ ---resp. $f_j(\theta+1/2-t)=-f_j(\theta+t)$.
\end{remark}

We are now in position to prove the following result:

\begin{theorem}\label{th:mz2cos}
Let $p\in[1,+\infty]$.
Let $\theta_1,\theta_2\in[0,1)$ be such that $\theta_1-\theta_2\notin\Q$ and $\Lambda,\Lambda'\subset\N$. 
If $\Lambda,\Lambda'$ are $[-1,1]$-M\"untz-Sz\'asz sequences, then
the system $\{\cos^\lambda 2\pi(t-\theta_1),\lambda\in\Lambda\}\cup\{\cos^{\lambda'} 2\pi(t-\theta_2),\lambda'\in\Lambda'\}$
is total in $X_p(0,1)$.
\end{theorem}

We will give two proofs of this theorem. The first one is ``constructive'' while the second one relies on the Hahn-Banach theorem.
The advantage of the second one is that it is more illustrative for our conjecture.

\begin{proof}[Direct proof] Let $f\in X_p(0,1)$ and $\eps>0$. There exists a trigonometric polynomial
such that $\norm{f-P}_p<\eps$ (such polynomials can be given explicitly via Fej\'er sums).
Write $P=P_1+P_2$ with $P_1,P_2$ given by Lemma \ref{lem:decomp} (those are again explicit).

Finally, as $\Lambda,\Lambda'$ are $[-1,1]$-M\"untz-Sz\'asz sequences, $P_1\in\tt_{\infty,\Lambda,\theta_1}$ and 
$P_2\in\tt_{\infty,\Lambda',\theta_2}$.
Therefore, there exists $Q_1,Q_2$ two polynomials such that, if we set $\pi_j=Q_j\bigl(\cos 2\pi(t-\theta_j)\bigr)$,
then $\norm{P_j-\pi_j}_p<\eps$ ($Q_1,Q_2$ can be explicitly given with the help of the constructive proofs of the M\"untz-Sz\'asz
theorem). But then 
$$
\norm{f-\pi_1-\pi_2}_p\leq \norm{f-P}_p+\norm{P_1-\pi_1}_p+\norm{P_2-\pi_2}_p<3\eps
$$
as expected.
\end{proof}

\begin{proof}[Indirect proof] If $\Lambda,\Lambda'$ are $[-1,1]$-M\"untz-Sz\'asz sequences and 
$f\in (\tt_{p,\Lambda,\theta_1}+\tt_{p,\Lambda',\theta_2})^\perp$ then 
$f\in \tt_{p,\Lambda,\theta_1}^\perp\cap\tt_{p,\Lambda',\theta_2}^\perp$.
Lemma \ref{lem:oneshift} then implies that
$$
\left\{\begin{matrix}f(\theta_1+t)&+&f(\theta_1-t)&=&0\\
f(\theta_2+t)&+&f(\theta_2-t)&=&0
\end{matrix}\right..
$$
This implies that $f=0$ ({\it see e.g.} \cite{Sj,Le}). For sake of completeness, here is a simple proof:
looking at Fourier coefficients, this system is equivalent to
$$
\left\{\begin{matrix}
e^{2i\pi k\theta_1}c_k(f)&+&e^{-2i\pi ki\theta_1}c_{-k}(f)&=&0\\
e^{2i\pi k\theta_2}c_k(f)&+&e^{-2i\pi ki\theta_2}c_{-k}(f)&=&0
\end{matrix}\right.\qquad,\forall k\in\Z.
$$
When $k=0$ we get $c_0(f)=0$. For $k\not=0$,
the system has determinant $2i\sin 2k\pi(\theta_1-\theta_2)\not=0$ since $\theta_1-\theta_2\notin\Q$.
Therefore $c_k(f)=0$ for every $k\in\Z$ and thus $f=0$.
\end{proof}

We conjecture that the reverse of this theorem is true as well:

\begin{conjecture}
Let $p\in[1,+\infty]$.
Let $\theta_1,\theta_2\in[0,1)$ be such that $\theta_1-\theta_2\notin\Q$ and $\Lambda,\Lambda'\subset\N$. 
If the system $\{\cos^\lambda 2\pi(t-\theta_1)\,:\lambda\in\Lambda\}\cup\{\cos^{\lambda'} 2\pi(t-\theta_2)\,:\lambda'\in\Lambda'\}$
is total in $X_p(0,1)$, then
$\Lambda,\Lambda'$ are $[-1,1]$-M\"untz-Sz\'asz sequences.
\end{conjecture}

We can now prove the following partial version of the conjecture:

\begin{theorem}
Let $p\in[1,+\infty]$ and $\Lambda\subset\N$.
Let $\theta_1,\theta_2\in[0,1)$ be such that $\theta_1-\theta_2$ is irrational and $a$-approximable for some $a>0$. 
If the system $\{\cos^{\lambda} 2\pi(t-\theta_1),\cos^\lambda 2\pi(t-\theta_2)\,:\lambda\in\Lambda\}$
is total in $X_p(0,1)$, then
$\Lambda$ is a $[-1,1]$-M\"untz-Sz\'asz sequence.
\end{theorem}

\begin{proof}
If $\Lambda$ is not a $[-1,1]$-M\"untz-Sz\'asz sequence, then at least one of the series 
$\sum_{\lambda\in\Lambda_e}\lambda^{-1}$ or $\sum_{\lambda\in\Lambda_o}\lambda^{-1}$
diverges, say the first one.

Let $j>a+1/2$ and $f$ be a $\cc^j$ $1$-periodic function with zero-mean and assume that $f$
is not analytic in at least $5$ points. Assume further that $c_k(f)=0$ when $k$ is odd.

Write $f=f_1+f_2$ be the decomposition given by Lemma \ref{lem:decom}. Note that, from the remark following its proof,
$f_1,f_2$ satisfy $f_\ell(\theta_\ell+1/2-t)=f_\ell(\theta_\ell+t)$, $\ell=1,2$.

Now, as functions in $\tt_{p,\Lambda,\theta_2}$ satisfy $\ffi(\theta_2-t)=\ffi(\theta_2+t)$, and
$f_1(\theta_1-t)=f_1(\theta_1+t)$, necessarily, $f_1\in \tt_{p,\Lambda,\theta_1}$ since otherwise $f_1$ would be constant.
Similarly, $f_2\in \tt_{p,\Lambda,\theta_2}$.
From Lemma \ref{lem:analytic} we therefore know that $f_1,f_2$ are analytic except possibly at two points each.
Finally $f=f_1+f_2$ is analytic except at $4$ points, a contradiction.
\end{proof}

\section{A link with Heisenberg Uniqueness Pairs}

The original idea behind this paper comes from an other problem, namely the notion
of Heisenberg Uniqueness Pairs recently introduced by Hedenmalm and Montes-Rodr\'iguez \cite{HMR}
and further investigated {\it e.g.} in \cite{JK,Le,Sj}:

\begin{definition} 
Let $\Lambda\subset\R^2$ and $\Gamma$ a smooth curve. Then
$(\Gamma,\Lambda)$ is a \emph{Heisenberg Uniqueness Pair} if the only finite measure
$\mu$ that is supported on $\Gamma$, absolutely continuous with respect to arc length on $\Gamma$ and such that
$\widehat{\mu}\Big|_{\Lambda}=0$ is the measure $\mu=0$.
\end{definition}

Take for instance $\Gamma=\{(\cos2\pi t,\sin2\pi t),\ t\in[0,1)\}$ to be the unit circle and
$\Lambda$ to be a set of two lines through the origin, $\Lambda=\{(t\cos\theta_1,t\sin\theta_1),\ t\in\R\}\cup
\{(t\cos 2\pi\theta_1,t\sin2\pi\theta_1),\ t\in\R\}$, $\theta_1\not=\theta_2\in[0,1)$.

Then Lev \cite{Le} and Sj\"olin \cite{Sj} independently proved that $(\Gamma,\Lambda)$ is a  Heisenberg Uniqueness Pair
if and only if $\theta_1-\theta_2$ is irrational.

But, a measure $\mu$ that is supported on $\Gamma$ and absolutely continuous with respect to arc length on $\Gamma$ is determined
by a density $f\in L^1(0,1)$ in the following way
$$
\scal{\mu,\ffi}=\int_0^1\ffi(\cos2\pi t,\sin2\pi t)f(t)\,\mbox{d}t\qquad f\in\cc(\R^2).
$$
In particular, $\dst\widehat{\mu}(\eta,\xi)=\int_0^1 f(t)e^{2i\pi (\eta\cos2\pi t+\xi\sin2\pi t)}\,\mbox{d}t$
so that $\widehat{\mu}=0$ on $\Lambda$ means that
$$
\int_0^1 f(t)e^{2i\pi r\cos2\pi (t-\theta_1)}\,\mbox{d}t=
\int_0^1 f(t)e^{2i\pi r\cos2\pi (t-\theta_2)}\,\mbox{d}t=0,\quad r\in\R.
$$
In particular, differentiating $\lambda$ times with respect to $r$ at $r=0$ leads to
\begin{equation}
\label{eq:last}
\int_0^1 f(t)\cos^\lambda2\pi (t-\theta_1)\,\mbox{d}t=
\int_0^1 f(t)\cos^\lambda2\pi (t-\theta_2)\,\mbox{d}t=0
\end{equation}
for each $\lambda\in\Lambda$.
From Theorem \ref{th:mz2cos}, if \eqref{eq:last} holds for every $\lambda$ in some $[-1,1]$-M\"untz-Sz\'asz set,
then $f=0$. Note that the fact that $f\in L^1$ and the fact that the circle $\Gamma$ is compact makes it
easy to justify all computations.

The first author's paper \cite{JK} contains many more examples of compact curves $\Gamma$ and pairs of lines
$\Lambda$ such that $(\Gamma,\Lambda)$ is a Heisenberg Uniqueness Pair. Each such example
leads to new pairs of functions for which the sufficient part of our M\"untz-Sz\'asz Theorem
\ref{th:mz2cos} holds. However, the converse seems much more difficult to establish.

\section*{Acknowledgments}
The first author kindly acknowledge financial support from the French ANR program, ANR-12-BS01-0001 (Aventures),
the Austrian-French AMADEUS project 35598VB - ChargeDisq, the French-Tunisian CMCU/UTIQUE project 32701UB Popart.
This study has been carried out with financial support from the French State, managed
by the French National Research Agency (ANR) in the frame of the Investments for
the Future Program IdEx Bordeaux - CPU (ANR-10-IDEX-03-02).

\end{document}